\documentclass[journal]{IEEEtran}
\usepackage{graphicx}
\usepackage{amsmath}
\usepackage{amssymb}
\usepackage{cite}

\newtheorem{theorem}{Theorem}[section]
\newtheorem{lemma}[theorem]{Lemma}

\newtheorem{remark}{Remark}
\newtheorem{assumption}{Assumption}

\newenvironment{pfof}[1]{\vspace{1ex}\noindent{\itshape Proof of
    #1:}\hspace{0.5em}} {\hfill\oprocend\vspace{1ex}}
\newenvironment{proof}[1]{\vspace{1ex}\noindent{\itshape Proof:}\hspace{0.5em}} {\hfill\oprocend\vspace{1ex}}

\renewcommand{\d}{\mathrm{d}}

\newcommand{\union}{\operatorname{\cup}}

\newcommand{\real}{\mathbb{R}}

\newcommand\oprocendsymbol{\hbox{$\square$}}
\newcommand\oprocend{\relax\ifmmode\else\unskip\hfill\fi\oprocendsymbol}

\newcommand{\vect}[1]{\mathbbold{#1}}
\newcommand{\vones}[1][]{\vect{1}_{#1}}
\newcommand{\vzeros}[1][]{\vect{0}_{#1}}
\DeclareSymbolFont{bbold}{U}{bbold}{m}{n}
\DeclareSymbolFontAlphabet{\mathbbold}{bbold}

\usepackage{enumerate}
\usepackage{xcolor}

\newcommand{\tb}{}

\graphicspath{{./Figures/}}

\ifCLASSINFOpdf
\else
\fi

\graphicspath{{./Figures/}}

\begin{document}
%
\title{Distributed Monitoring of Voltage Collapse Sensitivity Indices}
%
%
%

\author{John~W.~Simpson-Porco,~\IEEEmembership{Member,~IEEE,}
        and~Francesco~Bullo,~\IEEEmembership{Fellow,~IEEE}
\thanks{J.~W.~Simpson-Porco is with the Department of Electrical and Computer Engineering, University of Waterloo, Waterloo ON, N2L 3G1 Canada. Correspondance: {\texttt{jwsimpson@uwaterloo.ca}}. F.~Bullo is with the Center for Control, Dynamical Systems and Computation, University of California, Santa Barbara, CA 93106 USA. This work was supported by NSF CNS-1135819 and the Natural Science and Engineering Research Council of Canada.}
}

%
%

\markboth{}%
{Shell \MakeLowercase{\textit{et al.}}: Bare Demo of IEEEtran.cls for Journals}
%



\maketitle

\begin{abstract}
The assessment of voltage stability margins is a promising direction for wide-area monitoring systems. 
Accurate monitoring architectures for long-term voltage instability are typically centralized and lack scalability, while completely decentralized approaches relying on local measurements tend towards inaccuracy. 
Here we present {distributed linear} algorithms for the online computation of voltage collapse sensitivity indices. 
The computations are {collectively} performed by processors embedded at each bus in the smart grid, using synchronized phasor measurements and communication of voltage phasors {between neighboring} buses. 
{Our algorithms provably converge to the proper index values, as would be calculated using centralized information, but but do not require any central decision maker for coordination.} 
Modifications of the algorithms to account for generator reactive power limits are discussed. We illustrate the effectiveness of our designs with a case study of the New England 39 bus system.
\end{abstract}





\section{Introduction}
\label{Sec:Introduction}

Power grids are transitioning from a paradigm of centralized monitoring and control to one based on decentralized decisions and consumer interaction. 
When coupled with waining infrastructure investment, rapidly growing urban load centers, and the wide-spread adoption of intermittent distributed generation, this structural shift will lead to a broader and more uncertain range of operating conditions for the grid and an erosion of system stability margins if not properly coordinated. 

Simultaneously however, the rise of cheap and ubiquitous sensing, communication, and computational capability suggests a future where the physical grid is strongly coupled to many accompanying layers of data and control.
Unlike classic grids where sparsely available information is telemetered to a control center and processed, copious amounts of information are distributed throughout the smart grid along with the computational capabilities to process measured data and make coordinated decisions in real-time for wide-area monitoring, protection, and control (WAMPAC) \cite{VT-GV-DC-PR-VM-JF-SS-MMB-AP:11,AC-PPK:13}.
The distributed nature of these smart grid resources suggests we explore and evaluate the effectiveness of different information architectures for WAMPAC, ranging the spectrum from centralized to decentralized.

In this article we consider the problem of online monitoring for long-term voltage instability (LTVI) within a smart grid, which has recently been identified as an area of urgent interest for industry \cite{DN-VM-BB-VK-JC:08}.
While in general voltage stability is a complex, multi time-scale phenomena, long-term voltage instability is a quasi-static bifurcation instability \cite{ID:92,CAC:95} associated with an inability of the combined generation/transmission system to transmit sufficient power to loads \cite{TVC-CV:98}. After an increase in load, or a disturbance such as generation failure or load shedding, the grid's long-term operating point can vanish leading to the tripping of protection equipment and potentially to a large-scale voltage collapse blackout \cite{TVC-CV:98}. Robustness margins against LTVI are quantified via \emph{voltage collapse proximity indicators} (VCPIs), which produce measures of distance to instability \cite{CAC:02}. Accurate indicator estimates are key for distinguishing vulnerable system conditions from stable conditions exhibiting low voltages \cite{MG-TVC:11}, and we now review various architectures for the calculation of such indicators.

%


\subsection{Monitoring of Long-Term Voltage Instability}

Many monitoring solutions for LTVI have been proposed, and can be broadly classified by architecture (centralized, decentralized, distributed), measurement rate/complexity (time-skewed SCADA vs.\ time-stamped phasor measurement unit data), and theoretical rigor (heuristic vs.\ exact) \cite{MG-TVC:11}. The most important distinction for our purposes is the first, which we now expand on.

%

{
\textbf{Centralized Monitoring:} In a centralized monitoring approach, relevant data is telemetered to a central computer in a control center and potentially combined with state estimation to calculate relevant indices; see \cite{BM-MB:03,MG-TVC:11,MG-TVC:09a,MG-TVC:09b} and the many references therein. From the perspective of this paper, the main drawback of a centralized approach is that it results in a single point of failure for the monitoring system, and potentially requires data to be sent over large distances. Data privacy issues may also come into consideration. Moreover, in emerging applications such as microgrids, centralized supervision may be untenable or prohibitive. In this case, a more modular, scalable monitoring approach is desirable.
}


\smallskip

\textbf{Decentralized Monitoring:}\footnote{We use the term decentralized here for what is sometimes called ``completely decentralized'' or ``completely distributed'' \textemdash{} the VCPI calculated at bus $i$ will depend only on information measured locally at bus $i$, {such as phasor voltage $V_i\angle \theta_i$ and complex power injection $P_i + \mathrm{j}Q_i$}.} In contrast to centralized monitoring, decentralized techniques rely only on locally measured information to estimate voltage stability margins. Monitoring techniques based on PMU data and/or Th\'{e}venin equivalent circuits were proposed in \cite{VB-TSS-MSS:04,YG-NS-AG:06,SC-GNT:08,YS-XW:09,LR-ID:14}, among others. Decentralized VCPIs offer low implementation complexity and easy scalability, with the additional advantage that (typically) no  communication or state estimation is required. The price paid for these advantages is accuracy: decentralized VCPIs are invariably heuristics, often inspired by single-line power flow results, and are always too optimistic since they do not explicitly account for coupling between buses \cite{CDV-NGS:07}. {A notable exception for decentralized approaches is \cite{CDV-NGS:07}, where sensitivities of voltages with respect to on-load tap changer ratios are used to monitor the system for instability.}


A monitoring approach for LTVI was presented in \cite{LX-YC-HL:12}, where the grid is partitioned into overlapping monitoring areas (akin to control areas in automatic generation control). The indicators used however are somewhat unconventional, and the approach relies on selecting monitoring areas which are only weakly coupled. {Neglecting this inter-area coupling then results in a centralized assessment problem within each area.} While being termed ``distributed'', in our terminology the approach in \cite{LX-YC-HL:12} is a hybrid of centralized and decentralized ideas. {In Remark \ref{Rem:AreaExtension} we comment on the extension of our results to similar multi-area monitoring architecture.}

\smallskip

\textbf{Distributed Monitoring:} In the intermediate between centralized and decentralized we arrive at distributed monitoring strategies, which rely on local measurements along with communicated data from directly adjacent {buses or areas of the grid.} Distributed strategies promise to combine the performance of centralized monitoring with the scalability of decentralized monitoring, requiring only sparse, localized communication without any global information. Reflecting this, the recent literature has witnessed a steady growth in the applications of distributed algorithms to power system monitoring and control, now including resource allocation \cite{ADG-CNH:11}, load shedding \cite{YX-WL-JG:11},
economic dispatch \cite{FD-JWSP-FB:13y}, optimal power flow \cite{ED-HZ-GBG:13},
voltage control \cite{BAR-CNH-ADDG:13},
transfer capability assessment \cite{JHL-CCC:15}, and inverter coordination \cite{JWSP-QS-FD-JMV-JMG-FB:13s} 
in microgrids. To the authors knowledge however, distributed algorithms have not yet been designed for the monitoring of VCPIs. {In particular, we focus on a subclass of VCPIs termed sensitivity indices, which quantify the sensitivity of grid states to changes in grid parameters. These sensitivities increase as voltage collapse is approached, and monitoring these sensitivities therefore provides information on the proximity to collapse.
}

\subsection{Contributions}

{In this work we present the first distributed algorithms for the online computation of voltage collapse sensitivity indices}. We demonstrate that the \emph{exact} calculation of several standard centralized indices can be distributed among agents embedded within the smart grid, achieving centralized performance through only local measurements and short-range nearest-neighbor communication. The exact nature of these agents is left unconstrained; the software could be embedded in next generation power inverters, power electronic devices for voltage control, or implemented at generators or smart meters. Our algorithms do not rely on a state estimator processing sparse measurements to estimate state variables, but rather combine direct local measurements with measurements made and communicated by neighboring buses. Data is transmitted only over short distances, {minimizing} communication problems such as packet delays and measurement problems such as time-stamp drift. {No centralized decision maker is required.} Our approach does not rely on any pre-defined interfaces, on any representative sets of offline data used for learning, or on any Th\'{e}venin equivalent representations. {After algorithm convergence, each bus recovers its exact sensitivity index along with the indices of neighboring buses.} We demonstrate the efficacy of our algorithms via simulation in Section \ref{Sec:Sim} on the IEEE 39 bus system.


We assume that each bus in the system is equipped with a phasor measurement unit. While current power systems are not equipped with this level of observability, the smart grid eventually may, and demonstrations of the operational benefits of observability (such as those presented herein) will serve as incentive to invest in such measurement capabilities in the future. {It seems plausible that our assumption of full observability can be relaxed, and that the approach can be extended to more detailed models of long-term voltage instability, but we defer further discussion of this to our concluding remarks in Section \ref{Sec:Conclusions}}. At the transmission level, centralized, state estimation-based voltage stability margins will continue to play a major role, but it is nonetheless important to assess the advantages and limitations of alternative methods. {Our main message is that complicated sensitivity indices can in fact be computed using only localized information, without the need for centralized coordination or computation.} An area where our algorithms may prove particularly useful is microgrids, where centralized monitoring, control and optimization architectures are often absent and must be implemented collectively by coordinating devices within the microgrid in a scalable way.

\subsection{Preliminaries and Notation}
{\it Sets, vectors and functions:} We let $\real$ (resp. $\real_{>0}$) denote the set of real (resp. strictly positive real) numbers. Given $x \in \real^{n}$, $\|x\|_{\infty} = \max_{i\in\{1,\ldots,n\}} |x_i|$, and $[x] \in \real^{n\times n}$ is the associated diagonal matrix with $x$ on the diagonal.
%
Throughout, $\vones[n]$ and $\vzeros[n]$ are the $n$-dimensional vectors of unit and zero entries, and $\vzeros[]$ is a matrix of all zeros of appropriate dimensions. The $n \times n$ identity matrix is $I_n$.

\section{System Models and Sensitivity-Based Voltage Collapse Proximity Indicators}
\label{Sec:GridModel}

We begin by defining the grid models to be used in the paper before reviewing the relevant voltage collapse indices.

\subsection{Power System Model}

We model a balanced, quasi-synchronous power grid as a connected, undirected and weighted graph $(\mathcal{V},\mathcal{E})$, where $\mathcal{V}$ is the set of nodes (buses) and $\mathcal{E} \subseteq \mathcal{V} \times \mathcal{V}$ is the set of edges (branches).  We partition the set of buses $\mathcal{V}$ as $\mathcal{V} = \mathcal{L} \union \mathcal{G}$, with $n \geq 1$ load (PQ) buses $\mathcal{L} = \{1,\ldots,n\}$ and $m \geq 1$ generator (PV) buses $\mathcal{G} = \{n+1,\ldots,n+m\}$.\footnote{For our purposes, PV generator buses $\mathcal{G}$ may represent synchronous generators, frequency-dependent loads, or grid-forming inverters implementing droop controllers \cite{FD-JWSP-FB:13y,JWSP-FD-FB:13t}. Similarly, load buses $\mathcal{L}$ may represent standard static loads as well as inverters performing maximum power point tracking.}
Each branch $\{i,j\} \in \mathcal{E}$ is weighted by a transfer admittance $y_{ij} = g_{ij}+jb_{ij}$, where $g_{ij} \geq 0$ and $b_{ij} \leq 0$. 
We encode the weights and topology in the bus admittance matrix $Y$, with elements $Y_{ij} = -y_{ij}$ and $Y_{ii} = -\sum_{j=1}^{n+m}y_{ij} + y_{\mathrm{shunt},i}$, where $y_{\mathrm{shunt},i}$ is the shunt element at bus $i$. The conductance matrix $G$ and susceptance matrix $B$ are defined by $G = \mathrm{Re}(Y)$ and $B = \mathrm{Im}(Y)$.
To each bus we associate a phasor voltage $V_i\angle \theta_i$ and a complex power injection $P_i + j Q_i$, which are related by the \emph{power flow equations}
\begin{subequations}
\begin{align}
P_i &= \!\sum_{j\in\mathcal{V}} V_iV_j(G_{ij}\cos(\theta_i-\theta_j)+B_{ij}\sin(\theta_i-\theta_j))\,,
\label{Eq:Active}\\
Q_i &= \!\sum_{j\in\mathcal{V}} V_iV_j(G_{ij}\sin(\theta_i-\theta_j)-B_{ij}\cos(\theta_i-\theta_j))\,.
\label{Eq:Reactive}
\end{align}
\end{subequations}
The unknowns in \eqref{Eq:Active}--\eqref{Eq:Reactive} are the phase angles $\theta = (\theta_1,\ldots,\theta_{n+m})$ and the load voltages $V_L = (V_1,\ldots,V_n)$. {With shunt admittances absorbed into the admittance matrix, we assume that the remainder of the load at each PQ bus can be described by a constant power load model. Extending our approach to more general voltage-dependent static load models $P_i(V_i)$ and $Q_i(V_i)$ is straightforward, and requires only a few additional terms in the formulae and algorithms which follow; we omit these extensions for notational simplicity. Further comments on grid modeling are deferred to Section \ref{Sec:Conclusions}.}



\subsection{Cyber Layer Model}
\label{Sec:CyberModel}

We assume that devices (agents) are embedded at each bus which can measure local information, communicate information with devices at nearby buses, and perform basic computations on the measured data. 

{Regarding measurement, we assume that each bus $i \in \mathcal{V}$ is equipped with a PMU, yielding accurate synchronized measurements of voltage phasors $V_i\angle \theta_i$.}
In addition we assume that the processor at each bus has knowledge of (or access to measurements of) power system infrastructure \emph{incident to the bus}, such as local power consumption/generation ($P_i$, $Q_i$), the admittances of incident electrical lines ($y_{ij}$), and local shunt elements ($y_{\mathrm{shunt},i}$). Admittances may be know a priori, or estimated in an initialization phase using ranging technologies over power line communication (PLC) channels. Similarly, shunt susceptances designed to support voltage magnitudes are often either fixed or switched by local controllers which could be integrated into the processing equipment under consideration. {In contrast to SCADA systems which sample data every few seconds, PMUs under the IEEE Standard C37.118 \cite{KEM-DH-MGA-SA-MB-GB-GB-JB-JYC-BD-VG-BK-DK-AGP-JS-VS-JS-YS-CH-BK-EP:08,KEM:15} are synchronized and able to take 10's of samples per second. On the time-scales of interest for long-term voltage stability, this fast sampling is well approximated as continuous-time measurement. Throughout we assume high-quality measurements and do not distinguish between measured and true values: if PMU measurement quality is determined to be an issue, a distributed state-estimation and filtering layer \cite{VK-GBG:13,SB-RC-MT:14} can be implemented between the raw measurements and our algorithms to improve signal-to-noise ratios and reject bad data.
}


Regarding communication, we assume that each agent can communicate bidirectionally with the agents at adjacent buses to which it is electrically connected. Said differently, the topology of the communication layer mimics the physical grid topology. This communication could be achieved through power line communication (PLC), limited-range wireless, or Ethernet. {We emphasize that here our focus is not on detailed communication protocols, but on highlighting the sufficiency of local information exchange for the exact calculation of standard sensitivity indicators.} To streamline our mathematical developments throughout, we will therefore assume generous communication capabilities which in effect permit continuous-time communication. 
{Due to the large separation of time-scales between PMU sampling rates and LTVI, and the fact that our algorithms require only short-distance communications, throughout we assume that delays are negligible. Nonetheless, in Remark \ref{Rem:DistribSolvers} we comment on theoretical extensions to less restrictive communication assumptions.}

\subsection{Sensitivity-Based Voltage Collapse Proximity Indicators}
\label{Sec:VCPIReview}

As long-term voltage instability and voltage collapse is associated with saddle-node bifurcation of the power flow equations, singularity of the power flow Jacobian or related matrices has long been used as an indicator of voltage collapse \cite{VAV-VAS-VII:75}.
Related approaches include modal analysis, singular value and condition number indicators, 
sensitivity indices, 
continuation methods,
optimization, and energy VCPIs.
%
Surveys, classifications, and comparative studies of various VCPIs are available in \cite{RAS-AGC-JES-HLF:88,JB-PR:94,CAC-ACZDS-VHQ:96,MMEK-SA-MSK:97,TVC-CV:98,MMS-EMS-AMSM:99,AKS-DG:00,CAC:02,MVS-CKB:09,MG-TVC:11,ME-MS:13}.

Here we focus on one of the oldest classes of VCPIs, the sensitivity indices, which are based on the sensitivity of the system operating point to variations in parameters. The idea is that small variations in, for example, load demands, will produce large variations in bus voltages near bifurcation \cite{BMW-BRC:69,SA-AI:76,RAS-AGC-JES-HLF:88}. While many sensitivity-based VCPIs have been superseded in practical power system operations by more accurate, more computationally intensive techniques, they nonetheless provide intuitive actionable information, and are relatively straightforward to {define and interpret}. We recall three basic indices \cite[Sec. 8.2.3]{JM-JWB-JRB:08} and then comment on the information needed to compute them. 

\smallskip

\paragraph{The ``$\d V/\d Q$'' Index} This index measures the sensitivity of load bus voltage magnitudes with respect to changes in reactive power demands. For a multi-bus network, we may formulate the appropriate indicator $\mathbb{I}_i$ as
\begin{equation}\label{Eq:dvdq}
\mathbb{I}_{i} \triangleq \sum_{j \in \mathcal{L}}\nolimits\frac{Q_j}{V_i}\frac{\delta V_i}{\delta Q_j}\,,\qquad i \in \mathcal{L}\,.
\end{equation}
The summands are the point elasticities of the voltage at load bus $i$ with respect to the reactive demand at load bus $j$. The sum then evaluates the \emph{total elasticity} of the voltage at bus $i \in \mathcal{L}$. The index ranges from 0 at open-circuit conditions to $+\infty$ when the system reaches the point of collapse. 

\smallskip

\paragraph{The ``$\d V_L/\d V_G$'' Index} Also called the ``$\d V/\d E$'' index, this index measures the sensitivity of load bus voltages to changes in generator network voltage set points. The appropriate multi-bus index $\mathbb{J}_i$ is
\begin{equation}\label{Eq:dvde}
\mathbb{J}_i \triangleq \sum_{k \in \mathcal{G}} \nolimits \frac{\delta V_i}{\delta V_k}\,,\qquad i \in \mathcal{L}\,.
\end{equation}
Near open-circuit conditions $\mathbb{J}_{i}$ should be near unity, indicating that changes in load voltages track changes in generator voltages with unity gain (a ``controllability'' property \cite{RAS-AGC-JES-HLF:88}). The index tends to $+\infty$ at the point of collapse. 

\smallskip

\paragraph{The $\mathrm{d}Q_G/\mathrm{d}Q_L$ Index} This index measures the incremental reactive power generation required to supply an incremental amount of additional load, and therefore quantifies the (inverse) efficiency of reactive power transport through the network. The appropriate multi-bus definition is
\begin{equation}\label{Eq:dqdq}
\mathbb{K}_{i} \triangleq \sum_{k \in \mathcal{G}}\nolimits \frac{\delta Q_k}{\delta Q_i}\,,\qquad i \in \mathcal{L}\,.
\end{equation}
With our sign conventions, the $\mathbb{K}_{i}$ ranges from $-1$ at open-circuit to $-\infty$ at bifurcation, indicating that the network transports reactive power inefficiently near voltage collapse.

%

\smallskip

The important observations regarding the indices \eqref{Eq:dvdq}--\eqref{Eq:dqdq} are (i) that the matrices of derivatives defining them are generally \emph{dense} matrices, and (ii) that the matrix elements take into account the global state of the network. For example, for a processor at bus $i \in \mathcal{L}$ to directly compute $\mathbb{J}_i$, it would need to not only be directly aware of all generators connected to the network, but also know numerically how the set point $V_k$ of each influences the local voltage $V_i$. This sensitivity is in turn influenced by the presence (or absence) of loading/compensation at all other buses. Ostensibly then, \eqref{Eq:dvdq}--\eqref{Eq:dqdq} incorporate non-local information, and it would appear then that only an operator with centralized or near-centralized state information can evaluate them.


\section{Distributed Computation of Sensitivity-Based VCPIs}
\label{Sec:DistribVCPI}

We now detail our approach for distributing the computation of the VCPIs presented in Section \ref{Sec:VCPIReview}. We present our approach pedagogically for the $\d V/\d Q$ index \eqref{Eq:dvdq} before formally defining our distributed protocols for all indices \eqref{Eq:dvdq}--\eqref{Eq:dqdq} and addressing protocol convergence. To begin, note that in vector notation the $\mathrm{d}V/\mathrm{d}Q$ index \eqref{Eq:dvdq} becomes
$$
\mathbb{I}^{} = [V_L]^{-1}\frac{\delta V_L}{\delta Q_L}Q_L\,,
$$
where $\mathbb{I}^{} = (\mathbb{I}^{}_1,\ldots,\mathbb{I}^{}_n)$, $Q_L = (Q_1,\ldots,Q_n)$, $[V_L]$ is the diagonal matrix of load bus voltages, and $\delta V_L/\delta Q_L$ is the matrix with elements $\delta V_i/\delta Q_j$, $i,j \in \mathcal{L}$. Away from the point of collapse the matrix $\delta V_L/\delta Q_L$ is invertible, and we may equivalently write
\begin{equation}\label{Eq:SystemOfEquations}
\frac{\delta Q_L}{\delta V_L}[V_L]\mathbb{I}^{} = Q_L\,,
\end{equation}
which is a (dense) system of equations. Returning briefly to the power flow \eqref{Eq:Active}--\eqref{Eq:Reactive}, around an operating point $(\theta,V_L) \in \mathbb{R}^{n+m} \times \real^{n}_{>0}$ incremental changes $(\delta \theta, \delta V_L, \delta V_G)$ in phase angles, load voltages, and generator voltages are related to incremental changes $(\delta P, \delta Q_L, \delta Q_G)$ in active power injections and reactive power injections (load and generator) by
%
%
\begin{equation}\label{Eq:LinPowerFlow}
\begin{pmatrix}\delta P \\ \delta Q_L \\ \delta Q_G\end{pmatrix} = 
\left(
\begin{array}{c|c|c}
\frac{\partial P}{\partial \theta} & \frac{\partial P}{\partial V_L} & \frac{\partial P}{\partial V_G} \\ [0.5ex] \hline \\ 
[-2ex] \frac{\partial Q_L}{\partial \theta} & \frac{\partial Q_L}{\partial V_L} & \frac{\partial Q_L}{\partial V_G}
\\ [0.5ex] \hline \\
[-2ex] \frac{\partial Q_G}{\partial \theta} & \frac{\partial Q_G}{\partial V_L} & \frac{\partial Q_G}{\partial V_G}
\end{array}
\right)
\begin{pmatrix}\delta \theta \\ \delta V_L \\ \delta V_G\end{pmatrix}\,,
\end{equation}
where all partial derivatives are evaluated at the operating point. 
If variations in active power injections $\delta P$ and generator voltages $\delta V_G$ are held at zero, the first two block-rows of equations in \eqref{Eq:LinPowerFlow} may be solved to yield

\begin{equation}\label{Eq:TotalDerivative}
\frac{\delta Q_L}{\delta V_L} = \frac{\partial Q_L}{\partial V_L} - \frac{\partial Q_L}{\partial \theta}\left(\frac{\partial P}{\partial \theta}\right)^{\dagger} \frac{\partial P}{\partial V_L}\,.
\end{equation}
where $\dagger$ denotes the Moore-Penrose pseudoinverse (see Remark \ref{Rem:Ass}). The system of equations \eqref{Eq:SystemOfEquations} for $\mathbb{I}$ then becomes
$$
\frac{\partial Q_L}{\partial V_L}[V_L]\mathbb{I} - \frac{\partial Q_L}{\partial \theta}\left(\frac{\partial P}{\partial \theta}\right)^{\dagger} \frac{\partial P}{\partial V_L}[V_L]\mathbb{I} = Q_L\,.
$$
Introducing an auxiliary variable $\mathbb{I}_{\rm aux} \in \real^{n+m}$, this dense system of equations is equivalent to the expanded system
\begin{equation}\label{Eq:Pair}
\left(
\begin{array}{c|c}
\frac{\partial Q_L}{\partial V_L}[V_L] & \frac{\partial Q_L}{\partial \theta} \\ [0.5ex] \hline \\ [-2ex] \frac{\partial P}{\partial V_L}[V_L] & \frac{\partial P}{\partial \theta}
\end{array}\right)
\begin{pmatrix}\mathbb{I} \\ \mathbb{I}_{\rm aux}\end{pmatrix} = \begin{pmatrix}Q_L\\ \vzeros[n+m] \end{pmatrix}\,.
\end{equation} 
The coefficient matrix in \eqref{Eq:Pair} is sparse, its sparsity pattern closely related the physical grid topology. Indeed, the sparsity of such matrices has long been used as an aid for fast computation of stability margins \cite{PAL-TS-GA-DJH:92}. {\tb While sparsity of the unreduced Jacobian \eqref{Eq:LinPowerFlow} could also be exploited for computing the desired indices, \eqref{Eq:LinPowerFlow} will typically contain unnecessary information, which for our algorithms would lead to unnecessary communication and computation. For example, the third block-row in \eqref{Eq:LinPowerFlow} contains unnecessary information for the $\mathrm{d}V/\mathrm{d}Q$ index. We therefore find the reduced Jacobian-like matrix in \eqref{Eq:Pair} more useful to work with.} To propose the simplest, most intuitive distributed algorithm for calculating the stability index $\mathbb{I}$, we make the following assumptions.

\smallskip

\begin{assumption}[\textbf{System Matrix Stability}]\label{Ass:Jac}
All eigenvalues of the matrices
$$
\frac{\partial P}{\partial \theta}\,, \quad  \left(
\begin{array}{c|c}
\frac{\partial P}{\partial \theta} & \frac{\partial P}{\partial V_L}\\ [0.5ex] \hline \\ 
[-2ex] \frac{\partial Q_L}{\partial \theta} & \frac{\partial Q_L}{\partial V_L} \end{array}
\right)\,,\quad \left(
\begin{array}{c|c}
\frac{\partial P}{\partial \theta} & \frac{\partial P}{\partial V_L}[V_L]\\ [0.5ex] \hline \\ 
[-2ex] \frac{\partial Q_L}{\partial \theta} & \frac{\partial Q_L}{\partial V_L}[V_L] \end{array}
\right)\,.
$$
have positive real parts, with the exception of a simple zero eigenvalue for each with respective right eigenvectors $\vones[n+m]$, $(\vones[n+m],\vzeros[n])$ and $(\vones[n+m],\vzeros[n])$.
\end{assumption}

\begin{remark}[\textbf{Comments on Assumption \ref{Ass:Jac}}]\label{Rem:Ass}
The simple zero eigenvalues of the matrices in Assumption \ref{Ass:Jac} correspond to a uniform shift $\delta \theta \mapsto \delta \theta + \alpha \vones[n+m]$ of all phase angle deviations $\delta\theta$. Since phase is defined only up to a reference, this trivial degree of freedom may be removed by restricting $\mathbb{I}_{\rm aux}$ to lie in the subspace orthogonal to $\vones[n+m]$, in which case all three matrices are effectively invertible. In practice Assumption \ref{Ass:Jac} holds away from the point of collapse \cite{CAC-ACZDS-VHQ:96}, and the non-zero eigenvalues of these matrices are often found to be real or have small imaginary parts \cite[Appendix B.3]{CT:94}. Moreover, Assumption \ref{Ass:Jac} is quite natural since (1) the matrices under consideration describe stable small-signal behavior for certain classes of power system dynamics, and (2) these dynamics are known to not exhibit Hopf bifurcations, and hence the respective matrices can only become singular when an eigenvalue reaches the origin during saddle-node bifurcation \cite{CAC:95}. In this sense then, Assumption \ref{Ass:Jac} is ``necessary and sufficient'' for the linear system \eqref{Eq:Pair} to be well-posed. \hfill \oprocend
%
%
\end{remark}

Our key observation is that the matrix elements in \eqref{Eq:Pair} are determined by localized information: the $ij$th element depends only on the voltage phasors at buses $i$ and $j$ and on the admittances of the adjoining branches. It follows that by using phasor measurements and communication among adjacent buses, the solution of \eqref{Eq:Pair} for $(\mathbb{I},\mathbb{I}_{\rm aux})$ can be distributed among processors embedded at each bus. With this goal in mind, to each load bus $i \in \mathcal{L}$ we associate a pair of scalar states $(x_i,y_i) \in \real^2$, while to each generator bus $i \in \mathcal{G}$ we associate a scalar state $y_i \in \real$. We assume that these states can be communicated bidirectionally between directly adjacent buses. Our first formal result gives a simple distributed algorithm in continuous-time such that $\lim_{t\rightarrow \infty} x_i(t) \rightarrow \mathbb{I}_i$ for each $i \in \mathcal{L}$. For notational convenience we define the data coefficients
\begin{subequations}\label{Eq:Data}
\begin{align}
d_{ij} &\triangleq V_iV_j\left(G_{ij}\sin(\theta_i-\theta_j)-B_{ij}\cos(\theta_i-\theta_j)\right)\,,\\
D_{ij} &\triangleq V_iV_j\left(G_{ij}\cos(\theta_i-\theta_j)+B_{ij}\sin(\theta_i-\theta_j)\right)\,,
\end{align}
\end{subequations}
for each $i,j \in \mathcal{V}$, which depend only on known constants, locally measured PMU data, and PMU data communicated between adjacent buses.

\smallskip

\begin{theorem}[\textbf{Distributed $\boldsymbol{\d V/\d Q}$ Index}]\label{Thm:dvdq} Consider the $\d V/\d Q$ indices $\mathbb{I}_i$ defined in \eqref{Eq:dvdq} and let $d_{ij}$ and $D_{ij}$ be as in  \eqref{Eq:Data}. Let each load bus $i \in \mathcal{L}$ execute
\begin{subequations}\label{Eq:dvdq_update_load}
\begin{align}
\label{Eq:dvdq_update_load_x}
\tau\dot{x}_i &= Q_i(1-x_i)-P_iy_i-\sum_{j\in\mathcal{L}} d_{ij}x_j + \sum_{j\in\mathcal{V}} D_{ij}y_j\,,\\
\label{Eq:dvdq_update_load_y}
\tau\dot{y}_i &= Q_iy_i - P_ix_i - \sum_{j\in\mathcal{L}} D_{ij}x_j - \sum_{j\in\mathcal{V}}d_{ij}y_j \,,
\end{align}
\end{subequations}
for some chosen $\tau > 0$, while each generator bus $i \in \mathcal{G}$ executes
\begin{equation}\label{Eq:dvdq_update_gen}
\tau\dot{y}_i = Q_iy_i - \sum_{j\in\mathcal{L}} D_{ij}x_j - \sum_{j\in\mathcal{V}}d_{ij}y_j\,.
\end{equation}
Then for any initial condition $(x(0),y(0)) \in \real^{n}\times\real^{n+m}$ it holds for each $i \in \mathcal{L}$ that $\lim_{t \rightarrow \infty} x_i(t) = \mathbb{I}_i$.
\end{theorem}

\begin{proof}\,
Let $x = (x_1,\ldots,x_{n})$ and $y = (y_1,\ldots,y_{n+m})$ be the state vectors associated with  \eqref{Eq:dvdq_update_load}--\eqref{Eq:dvdq_update_gen}. Comparing the right-hand sides of \eqref{Eq:dvdq_update_load}--\eqref{Eq:dvdq_update_gen} with the power flow Jacobian matrix elements in Lemma \ref{Lem:Jacobian}, one finds that in vector notation \eqref{Eq:dvdq_update_load}--\eqref{Eq:dvdq_update_gen} reads as
\begin{equation}\label{Eq:dvdq_dynamics}
\tau \begin{pmatrix}\dot{x} \\ \dot{y}\end{pmatrix} = 
-\left(
\begin{array}{c|c}
\frac{\partial Q_L}{\partial V_L}[V_L] & \frac{\partial Q_L}{\partial \theta} \\ [0.5ex] \hline \\ [-2ex] \frac{\partial P}{\partial V_L}[V_L] & \frac{\partial P}{\partial \theta}
\end{array}\right)
\begin{pmatrix}x \\ y\end{pmatrix} + \begin{pmatrix}Q_L\\ \vzeros[n+m] \end{pmatrix}\,.
\end{equation}
{Conversely, \eqref{Eq:dvdq_update_load}--\eqref{Eq:dvdq_update_gen} are obtained by writing out \eqref{Eq:dvdq_dynamics} in components, using Lemma \ref{Lem:Jacobian} and the definitions of $d_{ij}$ and $D_{ij}$ in \eqref{Eq:Data}.} Comparing the dynamics \eqref{Eq:dvdq_dynamics} to the algebraic equation \eqref{Eq:Pair}, it follows that the steady-states of \eqref{Eq:dvdq_update_load}--\eqref{Eq:dvdq_update_gen} are one-to-one correspondence with the solutions $(\mathbb{I},\mathbb{I}_{\rm aux})$ of \eqref{Eq:Pair}. 
The system matrix in \eqref{Eq:dvdq_dynamics} a permutation of the third matrix in Assumption \ref{Ass:Jac}, and is therefore Hurwitz except for a simple eigenvalue at zero with associated right eigenvector $u_1 = (\vzeros[n],\frac{1}{\sqrt{n+m}}\vones[n+m])$. The component of the state which evolves parallel to $u_1$ only influences the (average value of the) auxiliary variable $y(t)$, and does not influence the index estimates $x(t)$. Since all other eigenvectors are associated with negative eigenvalues, it follows that $\lim_{t\rightarrow +\infty} x(t) = \mathbb{I}$, which completes the proof.
%
%
%
\end{proof}

The monitoring architecture is depicted in Figure \ref{Fig:Schematic} for a simple power system. While the sums in \eqref{Eq:dvdq_update_load}--\eqref{Eq:dvdq_update_gen} run over all loads or all buses, the coefficients $d_{ij}$ and $D_{ij}$ are zero when $\{i,j\}$ is not a physical branch of the network, and hence the only information needed at processor $i \in \mathcal{V}$ is that from adjacent buses. {Said differently, the proposed monitoring architecture requires only peer-to-peer communication, without centralized coordination.} Uniformity of the time-constant $\tau$ across all buses is formally required to infer stability, but as our case study in Section \ref{Sec:Sim} will demonstrate, nonuniform time-constants $\tau_i$ pose no difficulties when implemented. 

{We make five observations regarding the above algorithm. First, note that the storage and computational requirements for implementation are extremely low. Each agent stores only the local states $(x_i(t),y_i(t))$ or $y_i(t)$, and integrates an ordinary differential equation; storing the time-history of states is not required, nor is it required that each agent maintains an estimate of the entire algorithm state. Second, the method relies only on bus measurements, and no measurements of branch currents are required. Third, communication is required only between neighboring buses in the network, minimizing the effects of any communication delays. Fourth, the time variable $t$ in the algorithm should be interpreted as a computational time-scale; the time constant $\tau$ can be adjusted to achieve any desired convergence speed, limited ultimately by communication time scales, measurement sampling time, and system dynamics, but not by the algorithm itself. Fifth and finally, our method does not rely on a linearized power flow model; the linearity of \eqref{Eq:dvdq_update_load}--\eqref{Eq:dvdq_update_gen} comes from examining sensitivities of the nonlinear power flow \eqref{Eq:Active}--\eqref{Eq:Reactive}, with real-time measurements replacing a nonlinear power flow solver.}

\begin{figure}[t]
\begin{center}
\includegraphics[height=0.7\columnwidth]{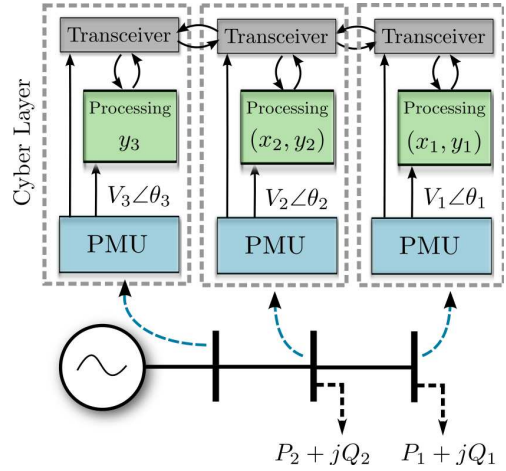}
\caption[]{Depiction of monitoring architecture for a radial three bus power system {implementing the algorithm \eqref{Eq:dvdq_update_load}--\eqref{Eq:dvdq_update_gen}}. Synchrophasor data is measured and fed to local processors, while communication between adjacent processors transfers both synchrophasor measurements $V_i\angle \theta_i$ and filter states $x_i,y_i$.}
\label{Fig:Schematic}
\end{center}
\end{figure}

\begin{remark}[\textbf{Relaxed Algorithms \& Communications}]\label{Rem:DistribSolvers}
As one can see from \eqref{Eq:dvdq_dynamics}, the distributed algorithm \eqref{Eq:dvdq_update_load}--\eqref{Eq:dvdq_update_gen} is of the form $\tau\dot{v} = -Av + b$, and hence can be explicitly discretized for distributed synchronous implementation as $v(k+1) = (I_{2n+m} - hA/\tau)v(k) + hb/\tau$ for a time-step $h > 0$. Under Assumption \ref{Ass:Jac}, this discrete-time system is stable if and only if $h < 2\tau \max_{i}\frac{\mathrm{Re}(\lambda_i)}{|\lambda_i|^2}$. Another option for distributed implementation is Jacobi iteration, where $A$ is decomposed into its diagonal part $T = \mathrm{diag}(a_{ii})$ and off diagonal part $R = A - T$, with the iteration taking the form $v(k+1) = T^{-1}(b-Rv(k))$. This iteration converges if and only if $\rho(T^{-1}R) < 1$; the authors have observed numerically that this assumption often holds for the relevant matrices in Assumption \ref{Ass:Jac}, but in general this assumption and Assumption \ref{Ass:Jac} are not equivalent. The easily verified diagonal dominance conditions for stability of the Jacobi iteration do not hold. 

Packets from neighbors may arrive asynchronously or with delays, and communication may be event-triggered based on sufficient changes in local measurements. While our focus is not on detailed communication protocols, we note a particular approach which is more complex but less restrictive. In \cite{SM-SAM:13} a discrete-time algorithm was developed for the distributed solution of linear equations such as \eqref{Eq:Pair}. In the proposed approach, each bus is assumed to know its respective row of the coefficient matrix, and updates a local estimate of the entire system state $(x,y)$ by exchanging estimates with neighbors. Thus, the requirements on information and communication are similar to the ones required by Theorem \ref{Thm:dvdq}, with slightly more local storage requirements. The approach extends to handle both asynchronism and delays \cite{JL-SM-ASM:13,SM-JL-SAM:15}. {When Assumption \ref{Ass:Jac} fails during extreme system conditions, the algorithm \eqref{Eq:dvdq_update_load}--\eqref{Eq:dvdq_update_gen} will diverge, and one of these alternatives would be required to calculate the relevant indices (which may even change sign under such conditions).}
%
%
%
%
%
%
\hfill \oprocend \end{remark}

{
\begin{remark}[\textbf{Extension to Multi-Area Monitoring}]\label{Rem:AreaExtension}
While we have presented the algorithm \eqref{Eq:dvdq_update_load}--\eqref{Eq:dvdq_update_gen} with a direct peer-to-peer implementation, it is easily extended to the case of multi-area monitoring. To see this, partition the the buses $\mathcal{V}$ of the network into $p \leq n+m$ non-overlapping monitoring areas $\mathcal{V} = \mathcal{A}_1 \cup \cdots \cup \mathcal{A}_p$. Depending on the specific problem setup, these areas could correspond to ISO regions, substations, phasor data concentrators (PDCs), or microgrids. 
Inside area $\mathcal{A}_k$, assume that a central processor $\mathcal{P}_k$ has (i) access to PMU measurements from each bus in area $\mathcal{A}_k$ (ii) knowledge of the grid topology and parameters in area $\mathcal{A}_k$ (iii) knowledge of the power lines which connect $\mathcal{A}_k$ to neighboring areas, and (iv) the ability to perform basic computations and communicate data with the processors in neighboring areas. 
In this case, the algorithm \eqref{Eq:dvdq_dynamics} would simply be block-partitioned according to the different areas, with central processors implementing the required blocks.
The case of one area $p = 1$ would correspond to complete centralized monitoring, where a central processor aggregates all information and performs all computations, while $p = n+m$ is the case described in main paper, where each bus (e.g., substation) constitutes an area, only local measurements are required, and information exchange is peer-to-peer.
Depending on regional data disclosure policies and privacy concerns, one architecture may be preferable over another; these issues are outside the scope of this work. \hfill \oprocend
\end{remark}
}

Similar filters to \eqref{Eq:dvdq_update_load}--\eqref{Eq:dvdq_update_gen} can be designed to calculate the $\d V_L/\d V_G$ index \eqref{Eq:dvde} and the $\d Q_G/\d Q_L$ index \eqref{Eq:dqdq}; the proofs may be found in Appendix \ref{App:1}.

\smallskip

\begin{theorem}[\textbf{Distributed $\boldsymbol{\d V_L/\d V_G}$ Index}]\label{Thm:dvde} Consider the $\d V_L/\d V_G$ indices $\mathbb{J}_i$ defined in \eqref{Eq:dvde} and let $d_{ij}$ and $D_{ij}$ be as in  \eqref{Eq:Data}. Let each load bus $i \in \mathcal{L}$ execute
\begin{subequations}
\begin{align*}
\tau\dot{x}_i &= -\frac{Q_i}{V_i}x_i-P_iy_i-\sum_{j\in\mathcal{L}} \frac{d_{ij}}{V_j}x_j + \sum_{j\in\mathcal{V}} D_{ij}y_j - \sum_{j \in \mathcal{G}}\frac{d_{ij}}{V_j}\,,\\
\tau\dot{y}_i &= Q_iy_i - \frac{P_i}{V_i}x_i - \sum_{j\in\mathcal{L}} \frac{D_{ij}}{V_j}x_j - \sum_{j\in\mathcal{V}}d_{ij}y_j - \sum_{j \in \mathcal{G}}\frac{D_{ij}}{V_j} \,,
\end{align*}
\end{subequations}
for some chosen $\tau > 0$, while each generator bus $i \in \mathcal{G}$ executes
\begin{equation*}
\tau\dot{y}_i = - \frac{P_i}{V_i} + Q_iy_i - \sum_{j\in\mathcal{L}} \frac{D_{ij}}{V_j}x_j - \sum_{j\in\mathcal{V}}d_{ij}y_j  - \sum_{j \in \mathcal{G}}\frac{D_{ij}}{V_j} \,.
\end{equation*}
Then for any initial condition $(x(0),y(0)) \in \real^{n}\times\real^{n+m}$ it holds for each $i \in \mathcal{L}$ that $\lim_{t \rightarrow \infty} x_i(t) = \mathbb{J}_i$.
\end{theorem}

\begin{theorem}[\textbf{Distributed $\boldsymbol{\d Q_G/\d Q_L}$ Index}]\label{Thm:dqdq} Consider the $\d Q_G/\d Q_L$ indices $\mathbb{K}_i$ defined in \eqref{Eq:dqdq} and let $d_{ij}$ and $D_{ij}$ be as in  \eqref{Eq:Data}. Let each load bus $i \in \mathcal{L}$ execute
\begin{subequations}
\begin{align*}
\tau\dot{x}_i &= -\frac{Q_i}{V_i}x_i-\frac{P_i}{V_i}(y_i-z_i)-\sum_{j\in\mathcal{L}} \frac{d_{ji}}{V_i}x_j\\
&\qquad \qquad \qquad - \sum_{j\in\mathcal{V}} \frac{D_{ji}}{V_i}(y_j-z_j) + \sum_{j \in \mathcal{G}}\frac{d_{ji}}{V_i}\,,\\
\tau\dot{y}_i &= Q_iy_i - P_ix_i + \sum_{j\in\mathcal{L}} D_{ji}x_j - \sum_{j\in\mathcal{V}}d_{ji}y_j\,,\\
\tau\dot{z}_i &= Q_iz_i - \sum_{j\in\mathcal{V}}d_{ji}z_j + \sum_{j \in \mathcal{G}}D_{ji}\,,
\end{align*}
\end{subequations}
for some chosen $\tau > 0$, while each generator bus $i \in \mathcal{G}$ executes
\begin{align*}
\tau\dot{y}_i &= Q_iy_i + \sum_{j\in\mathcal{L}} D_{ji}x_j - \sum_{j\in\mathcal{V}}d_{ji}y_j\,,\\
\tau\dot{z}_i &= -P_i + Q_iz_i - \sum_{j\in\mathcal{V}}d_{ji}z_j + \sum_{j \in \mathcal{G}}D_{ji}\,,
\end{align*}
Then for any initial condition $(x(0),y(0),z(0)) \in \real^{n}\times\real^{n+m}\times \real^{n+m}$ it holds for each $i \in \mathcal{L}$ that $\lim_{t \rightarrow \infty} x_i(t) = \mathbb{K}_i$.
\end{theorem}

\subsection{Incorporating Generator VAR Limits}
\label{Sec:OXL}

The distributed algorithms presented in Theorems \ref{Thm:dvdq}--\ref{Thm:dqdq} ignore an important factor in LTVI studies, namely the reactive power limitations of generators \cite{AEE-YHG:96,FD-BHC-MLC-LA:05,PAR-PWS:06}. When a synchronous generator exceeds these reactive power limits (derived from field and armature current limits) over medium time-scales, the AVR system becomes unable to regulate the network-side generator voltage and over-excitation limiters fix the reactive power output at its limit. On the long-time scales of interest for us, we can therefore approximate this behavior by replacing the PV bus model with a PQ bus model when the generator is at or above its reactive limit \cite{CT:94}. 

%
%

The approach for incorporating these limits into the distributed algorithm \eqref{Eq:dvdq_update_load}--\eqref{Eq:dvdq_update_gen} of Theorem \ref{Thm:dvdq} is as follows (similar approaches hold for the remaining two algorithms). If the reactive power supplied by generator $i \in \mathcal{G}$ satisfies $Q_i < Q_i^{\rm max}$, then the associated processor executes \eqref{Eq:dvdq_update_gen}, just as before. When $Q_i \geq Q_i^{\rm max}$, the processor initializes an additional internal state $x_i(t)$ and instead executes \eqref{Eq:dvdq_update_load_x}--\eqref{Eq:dvdq_update_load_y}. If necessary, this can be accompanied with a binary alert message to its neighbors signaling that a switch has taken place. To avoid chattering due to oscillating reactive power injections during transients, a temporal hysteresis can be used which ensures that $Q_i$ remains above or below $Q_i^{\rm max}$ for a sufficient amount of time before a switch in algorithm is made.

\subsection{Monitoring Thresholds and Worst-Case Indices}
\label{Sec:Consensus}

The algorithms in Theorems \ref{Thm:dvdq}--\ref{Thm:dqdq} give the processor at load bus $i \in \mathcal{L}$ a converging estimate of its stability index $\mathbb{I}_i, \mathbb{J}_i$ or $\mathbb{K}_i$, as well converging estimates of the same indices for any adjacent buses which are also load buses. Based on this information, we highlight two additional steps for monitoring that may be desirable. We discuss the $\d V/\d Q$ algorithm \eqref{Eq:dvdq_update_load}--\eqref{Eq:dvdq_update_gen}; similar statements apply to the other algorithms.

\paragraph*{Monitoring Thresholds} Suppose that to each load bus $i \in \mathcal{L}$ we associate a threshold $\gamma_i > 0$ for the index $\mathbb{I}_i$. These thresholds may be determined by experience, offline trials, or determined online by yet another distributed algorithm. If during monitoring $x_i(t)$ increases above $\gamma_i$ and remains there, an alert is triggered and communicated to neighboring processors. This in turn could trigger localized control responses, or the alert could be propagated system-wide.

\paragraph*{Global Knowledge of Worst-Case Index} Voltage stability of the network is ultimately limited by the weakest or most sensitive bus, quantified in our setup by the largest nodal value $\|\mathbb{I}\|_{\infty} = \max_{i \in \mathcal{L}}\mathbb{I}_i$ of the stability index. It may therefore be desirable for all buses to maintain an estimate $w_i(t)$ of the worst-case index $\|\mathbb{I}\|_{\infty}$ and continuously update it. A simple distributed protocol for achieving this is called \emph{max-consensus} \cite{FI-PC-JJ:12,SG-DDP-AP-AR:13} where each processor executes (in discrete-time)
\begin{equation}\label{Eq:Max_Consensus}
w_i(k+1) = \max\left\{w_i(k),\max_{j, \{i,j\}\in\mathcal{E}} w_j(k)\right\}\,,
\end{equation}
{with the initialization $w_i(t_0) = x_i(t_0)$ for $i \in \mathcal{L}$, where $t_0$ is the time at which execution begins. As generator buses $i \in \mathcal{G}$ do not carry a local state $x_i$, each $w_i(t_0)$ is initialized to a common value $w^*$ for each $i \in \mathcal{G}$, equaling the open-circuit value of the voltage stsability index under consideration. For example, for the $\mathrm{d}V/\mathrm{d}Q$ index $w^* = 0$, while for the $\mathrm{d}V_L/\mathrm{d}V_G$ index $w^* = 1$.} Each processor observes its own index and the indices of its neighbors and updates its estimate with the largest value it sees. By re-initializing and re-executing this periodically, all processors in the network can be made aware of the largest sensitivity.

\section{Case Study: IEEE 39 Bus System}
\label{Sec:Sim}

We demonstrate our approach by implementing our algorithm for the $\d V_L/\d V_G$ index of Theorem \ref{Thm:dvde} on a dynamic model of the reduced New England power grid, containing 9 generators and 30 load buses. A six-state two-axis model is used for the generators consisting of two-state mechanical dynamics, two-state electrical dynamics, {a single-state excitation system and a single-state governor with droop \cite{PWS-MAP:90}; generator and network parameters are drawn from \cite{MAP:89,PWS-MAP:90,RDZ-CEM-DG:11}.}

In place of a uniform filter time constant $\tau$, we let each processor implement its $\d V_L/\d V_G$ filter with a time constant $\tau_i$, which we draw from a uniform distribution between 10s and 20s. {Synchrophasor measurements are assumed to be corrupted with uncorrelated zero mean Gaussian noise, with standard deviation 0.001 p.u. on voltage magnitudes (arising from quantization and harmonic distortion), and $0.01^{\circ}$ on phase angles (due to sampling time discrepancies and inexact synchronization). At a $2\sigma$ level, these values are in compliance with the maximum total phasor error of 1\% specified by IEEE Standard C37.118-2011 \cite{KEM:15}.} Beginning from the base load case \cite{RDZ-CEM-DG:11}, power demands are ramped {along the base case} by 15\% between $t=20$s and $t=40$s, with the newly ramped load being shed abruptly at $t=200$s. 

{The filter states $x_i(t)$ are plotted in Figure \ref{Fig:Sim1} for load buses 3, 12, and 20, and the generator bus 34 whose state is initialized as $x_{34}(0) = 1$. The exact steady-state values of the respective stability indices are plotted in dashed black for each bus, as computed by a central processor solving the linear equation \eqref{Eq:dvde_vector} at each moment in time. First, we observe that the algorithm is able to accurately track the ramp in load between 20s and 40s. The estimates for buses 3, 12, and 20 have effectively converged to their proper values shortly before $t=50$s, but the increase in load has caused the generator at bus 34 to hit its reactive power limits. The respective estimator $x_{34}(t)$ comes online at roughly $t=54$s and converges rapidly, contributing to a further increase in the index estimates $x_i(t)$ of all other buses, and in particular at bus 20 which is directly adjacent to bus 34. At $t=200$s the excess load is shed and filter estimates converge back to their original values; the centralized computation displays significant ringing due to transient dynamics, while the filter state converges relatively smoothly due to its natural first-order dynamics, which act as a low-pass filter. The generator falls back below its reactive power limits, and after an anti-chattering delay the estimator for bus 34 is reset.}

\begin{figure}[t]
\begin{center}
\includegraphics[trim=20mm 0mm 0mm 0mm,clip=true,height=0.75\columnwidth]{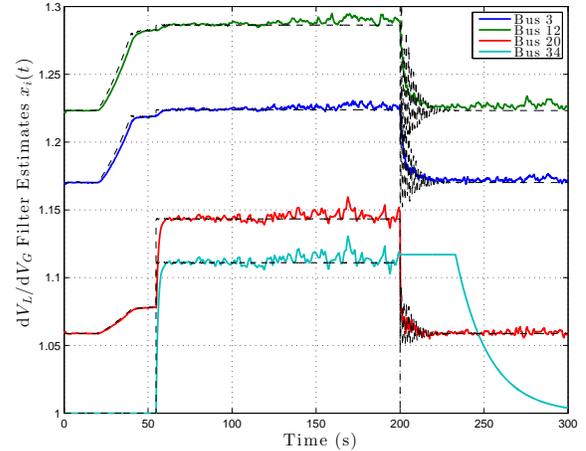}
\caption[]{Estimates $x_i(t)$ of the $\d V_L/\d V_G$ indices for several buses. Black dashed lines denote the exact index values for the respective buses, as calculated in a centralized manner. Noise is omitted on the centralized calculation for clarity.}
\label{Fig:Sim1}
\end{center}
\end{figure}

\begin{figure}[t]
\begin{center}
\includegraphics[trim=22mm 0mm 0mm 0mm,clip=true,height=0.7\columnwidth]{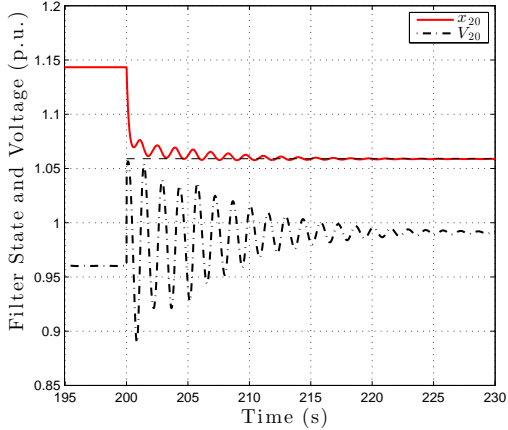}
\caption[]{Estimate $x_{20}(t)$ of the $\d V_L/\d V_G$ index along with the bus voltage at bus 20 for $t \in [195,230]$. The black dashed line denotes the exact index value. Noise is omitted for clarity.}
\label{Fig:Sim1_Close}
\end{center}
\end{figure}

{Figure \ref{Fig:Sim1_Close} displays a close-up of the trace of $x_{i}(t)$ at bus 20 after the load is shed at $t = 200$s, along with the resulting dynamics of the corresponding bus voltage $V_{20}(t)$. As our algorithms use real-time measurements for computing the sensitivity indices, transients experienced by the physical bus variables also impact the filter estimates until convergence occurs. As can be seen from Figure \ref{Fig:Sim1_Close} however, transients in physical variables tend to be damped significantly by the filter.}

{Figure \ref{Fig:Sim2} shows the output for all buses of the max-consensus iteration  \eqref{Eq:Max_Consensus}. Every 30 seconds, generator states were initialized at $w_i(t_0) = 1$, while load bus states were initialized at $x_i(t_0)$. Iterations were performed once a second, and within four to five iterations each bus converges to the largest bus sensitivity. Comparing Figure \ref{Fig:Sim2} to Figure \ref{Fig:Sim1}, this largest sensitivity can be seen to belong to bus 12, with the trend in Figure \ref{Fig:Sim2} accurately tracking the green trace of Figure \ref{Fig:Sim1}. Each bus therefore quickly obtains knowledge of the worst-case global sensitivity.}

\begin{figure}[t]
\begin{center}
\includegraphics[trim=20mm 0mm 0mm 0mm,clip=true,height=0.75\columnwidth]{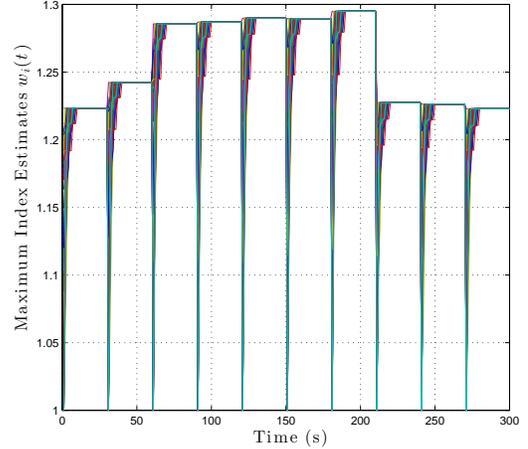}
\caption[]{Estimates $w_i(t)$ of $\max_{i \in \mathcal{V}} x_i(t)$ generated from the max-consensus algorithm \eqref{Eq:Max_Consensus}. The algorithm is reinitialized with the current value of $x_i(t)$ and re-executed every 30 seconds. Generator states are initialized from unity.}
\label{Fig:Sim2}
\end{center}
\end{figure}






\section{Conclusions \& Future Work}
\label{Sec:Conclusions}

{We have presented algorithms which distribute the computation of sensitivity-based voltage stability indices among processors embedded within a smart grid. Using PMU measurements and communication with adjacent buses, each bus is able to iteratively calculate its exact voltage stability sensitivity index. This exact computation is achieved without the requirement of a central decision maker, and we have illustrated the feasibility of the approach via simulation study.

While in this work we have used a purely power flow-based model to calculate sensitivities, an important extension is to augment the power flow equations \eqref{Eq:Active}--\eqref{Eq:Reactive} with additional equilibrium equations \cite[Chapter 6]{TVC-CV:98} corresponding to {\tb generator, controller, limiter, and load  dynamics. Moreover, steady-state reactive power limits are typically functions of active power injections.} As these additional equilibrium equations are typically not coupled between buses, designing analogous distributed algorithms should be feasible {\tb as inter-bus coupling occurs only through the power flow equations}.}
While we have included a basic load-flow type model of generator reactive power limits, more detailed models could be incorporated which include limits based directly on field and armature currents. Future work will also concern relaxing the assumption of full grid observability to partial observability by combining the algorithms presented herein with distributed estimation. Another key direction is the fusion of monitoring protocols such as the one presented herein with automatic distributed control and protection. It seems plausible that optimal control directions in parameter space \cite{ID-LL:92b} can be computed and executed in a distributed way, creating a self-healing grid.


\appendices

\section{Technical Lemmas and Proofs}
\label{App:1}

\begin{lemma}[\textbf{Power Flow Jacobian}]\label{Lem:Jacobian}
Let $(\theta,V_L) \in \real^{n+m}\times\real^n_{>0}$ be a solution of the power flow equations \eqref{Eq:Active}--\eqref{Eq:Reactive}, and let $d_{ij}$ and $D_{ij}$ be as in \eqref{Eq:Data}. When evaluated at at $(\theta,V_L)$ the partial derivatives of \eqref{Eq:Active}--\eqref{Eq:Reactive} are given (for $j \neq i$) by \cite[Sec. 3.5.1.1]{JM-JWB-JRB:08}
\begin{align*}
\frac{\partial P_i}{\partial \theta_i} &= -Q_i + d_{ii}\,,\quad \frac{\partial P_i}{\partial \theta_j} = d_{ij}\,,\\
%
%
V_i\frac{\partial P_i}{\partial V_i} &= P_i + D_{ii}\,,\quad V_j\frac{\partial P_i}{\partial V_j} = D_{ij}\,,\\
V_i\frac{\partial Q_i}{\partial V_i} &= Q_i + d_{ii}\,,\quad V_j\frac{\partial Q_i}{\partial V_j} = d_{ij}\,,\\
\frac{\partial Q_i}{\partial \theta_i} &= P_i - D_{ii}\,,\quad \frac{\partial Q_i}{\partial \theta_j} = -D_{ij}\,.
\end{align*}
\end{lemma}

\begin{pfof}{Theorem \ref{Thm:dvde}}
To begin, note that by setting $\delta P = \vzeros[n+m]$ and $\delta Q_L = \vzeros[n]$ in \eqref{Eq:LinPowerFlow} and eliminating $\delta \theta$ from the first two blocks of equations, one obtains
$$
\vzeros[n] = \frac{\delta Q_L}{\delta V_L}\delta V_L + \frac{\delta Q_L}{\delta V_G}\delta V_G\,,
$$
where $\delta Q_L/\delta V_L$ is as in \eqref{Eq:TotalDerivative} and
$$
\frac{\delta Q_L}{\delta V_G} = \frac{\partial Q_L}{\partial V_G} - \frac{\partial Q_L}{\partial \theta}\left(\frac{\partial P}{\partial \theta}\right)^{\dagger}\frac{\partial P}{\partial V_G}\,.
$$
It follows then from the definition \eqref{Eq:dvde} of the index $\mathbb{J}_i$ that
\begin{equation}\label{Eq:dvde_vector}
\mathbb{J} = -\left(\frac{\delta Q_L}{\delta V_L}\right)^{-1}\frac{\delta Q_L}{\delta V_G}\vones[m]\,,
\end{equation}
where $\mathbb{J} = (\mathbb{J}_1,\ldots,\mathbb{J}_n)$. Using Lemma \ref{Lem:Jacobian} one may deduce that the distributed algorithm in Theorem \ref{Thm:dvde} may be written in vector form as
\begin{equation}\label{Eq:dvde_dynamics}
\tau \begin{pmatrix}\dot{x} \\ \dot{y}\end{pmatrix} = 
-\left(
\begin{array}{c|c}
\frac{\partial Q_L}{\partial V_L} & \frac{\partial Q_L}{\partial \theta} \\ [0.5ex] \hline \\ [-2ex] \frac{\partial P}{\partial V_L} & \frac{\partial P}{\partial \theta}
\end{array}\right)
\begin{pmatrix}x \\ y\end{pmatrix} - 
\left(
\begin{array}{c}
\frac{\partial Q_L}{\partial V_G}\vones[m] \\ [0.5ex] \hline \\ [-2ex] \frac{\partial P}{\partial V_G}\vones[m] \end{array}\right)\,.
\end{equation}
Setting the left-hand side of \eqref{Eq:dvde_dynamics} to zero and eliminating the auxiliary state $y$, one finds that the unique $x$-component of any equilibrium is given uniquely by $x = \mathbb{J}$. Convergence of $x(t)$ to $\mathbb{J}$ follows from arguments similar to those in the proof of Theorem \ref{Thm:dvdq}.
\end{pfof}

\begin{pfof}{Theorem \ref{Thm:dqdq}}
To begin, set $\delta P = \vzeros[n+m]$ and $\delta V_G = \vzeros[m]$ in \eqref{Eq:LinPowerFlow} and eliminate $\delta \theta$ from the second and third blocks of equations to obtain
\begin{align*}
\delta Q_L &= \frac{\delta Q_L}{\delta V_L} \delta V_L\,,\quad \delta Q_G = \frac{\delta Q_G}{\delta V_L} \delta V_L\,,
\end{align*}
where $\delta Q_L/\delta V_L$ is as in \eqref{Eq:TotalDerivative} and
$$
\frac{\delta Q_G}{\delta V_L} = \frac{\partial Q_G}{\partial V_L} - \frac{\partial Q_G}{\partial \theta}\left(\frac{\partial P}{\partial \theta}\right)^{-1}\frac{\partial P}{\partial V_L}\,.
$$
Eliminating $\delta V_L$ from this pair, we find that
$$
\frac{\delta Q_G}{\delta Q_L} = \frac{\delta Q_G}{\delta V_L}\left(\frac{\delta Q_L}{\delta V_L}\right)^{-1}\,.
$$
Comparing to the definition of the index $\mathbb{K}_i$ in \eqref{Eq:dqdq}, we find that in vector form
$$
\mathbb{K} = \left(\frac{\delta Q_G}{\delta Q_L} \right)^{\sf T}\vones[m] = \left(\frac{\delta Q_L}{\delta V_L}\right)^{\sf -T}\left(\frac{\delta Q_G}{\delta V_L}\right)^{\sf T} \vones[m]\,,
$$
where $\mathbb{K} = (\mathbb{K}_1,\ldots,\mathbb{K}_n)$. Using Lemma \ref{Lem:Jacobian} one may deduce that the distributed algorithm in Theorem \ref{Thm:dvde} may be written in vector form as
\begin{equation}
\begin{aligned}\label{Eq:dqdq_vector}
\tau\begin{pmatrix}\dot{x} \\ \dot{y} \\ \dot{z}\end{pmatrix} &= 
-\left(
\begin{array}{c|c|c}
\left(\frac{\partial Q_L}{\partial V_L}\right)^{\sf T} & \left(\frac{\partial P}{\partial V_L}\right)^{\sf T} & -\left(\frac{\partial P}{\partial V_L}\right)^{\sf T}
\\ [0.5ex] \hline \\ 
[-2ex] \left(\frac{\partial Q_L}{\partial \theta}\right)^{\sf T} & \left(\frac{\partial P}{\partial \theta}\right)^{\sf T} & \vzeros[]
\\ [0.5ex] \hline \\
[-2ex] \vzeros[] & \vzeros[] & \left(\frac{\partial P}{\partial \theta}\right)^{\sf T}
\end{array}
\right)
\begin{pmatrix}x \\ y \\ z\end{pmatrix}\\
\,\\
&\quad + \begin{pmatrix}\vones[m]^{\sf T}\frac{\partial Q_G}{\partial V_L} & \vzeros[n+m]^{\sf T} & -\vones[m]^{\sf T}\frac{\partial Q_G}{\partial\theta}\end{pmatrix}^{\sf T}\,.
\end{aligned}
\end{equation}
By setting the left-hand side of \eqref{Eq:dqdq_vector} to zero and eliminating $y$ and $z$, one may verify that the $x$-component of any equilibrium is given uniquely by $x = \mathbb{K}$.
Convergence of $x(t)$ to the index $\mathbb{K}$ follows from arguments similar to those in the proof of Theorem \ref{Thm:dvdq}.
\end{pfof}

%

%
\IEEEpeerreviewmaketitle


\ifCLASSOPTIONcaptionsoff
  \newpage
\fi


\bibliographystyle{IEEEtran}
\bibliography{alias,Main,FB}

\begin{IEEEbiography}[{\includegraphics[width=1in,height=1.25in,clip,keepaspectratio]{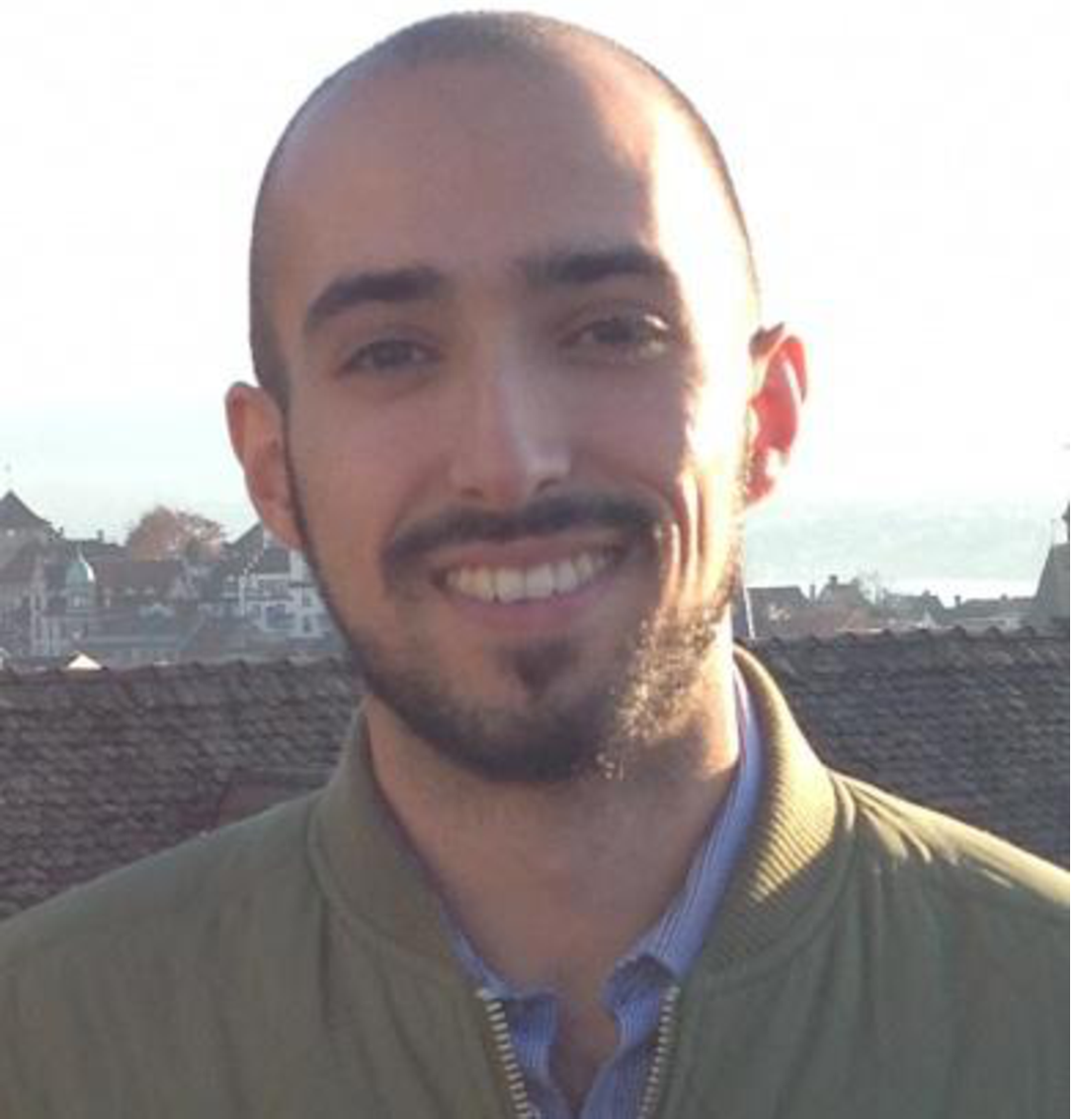}}]{John W. Simpson-Porco} (S'11--M'16) received the B.Sc. degree in engineering physics from QueenÕs University, Kingston, ON, Canada in 2010, and the Ph.D. degree in mechanical engineering from the University of California at Santa Barbara, Santa Barbara, CA, USA in 2015.

He is currently an Assistant Professor of Electrical and Computer Engineering at the University of Waterloo, Waterloo, ON, Canada. He was previously a visiting scientist with the Automatic Control Laboratory at ETH Z\"{u}rich, Z\"{u}rich, Switzerland. His research focuses on the control and optimization of multi-agent systems and networks, with applications in modernized power grids.

Prof. Simpson-Porco is a recipient of the 2012--2014 IFAC Automatica Prize and the Center for Control, Dynamical Systems and Computation Outstanding Scholar Fellowship.
\end{IEEEbiography}

\begin{IEEEbiography}[{\includegraphics[width=1in,height=1.25in,clip,keepaspectratio]{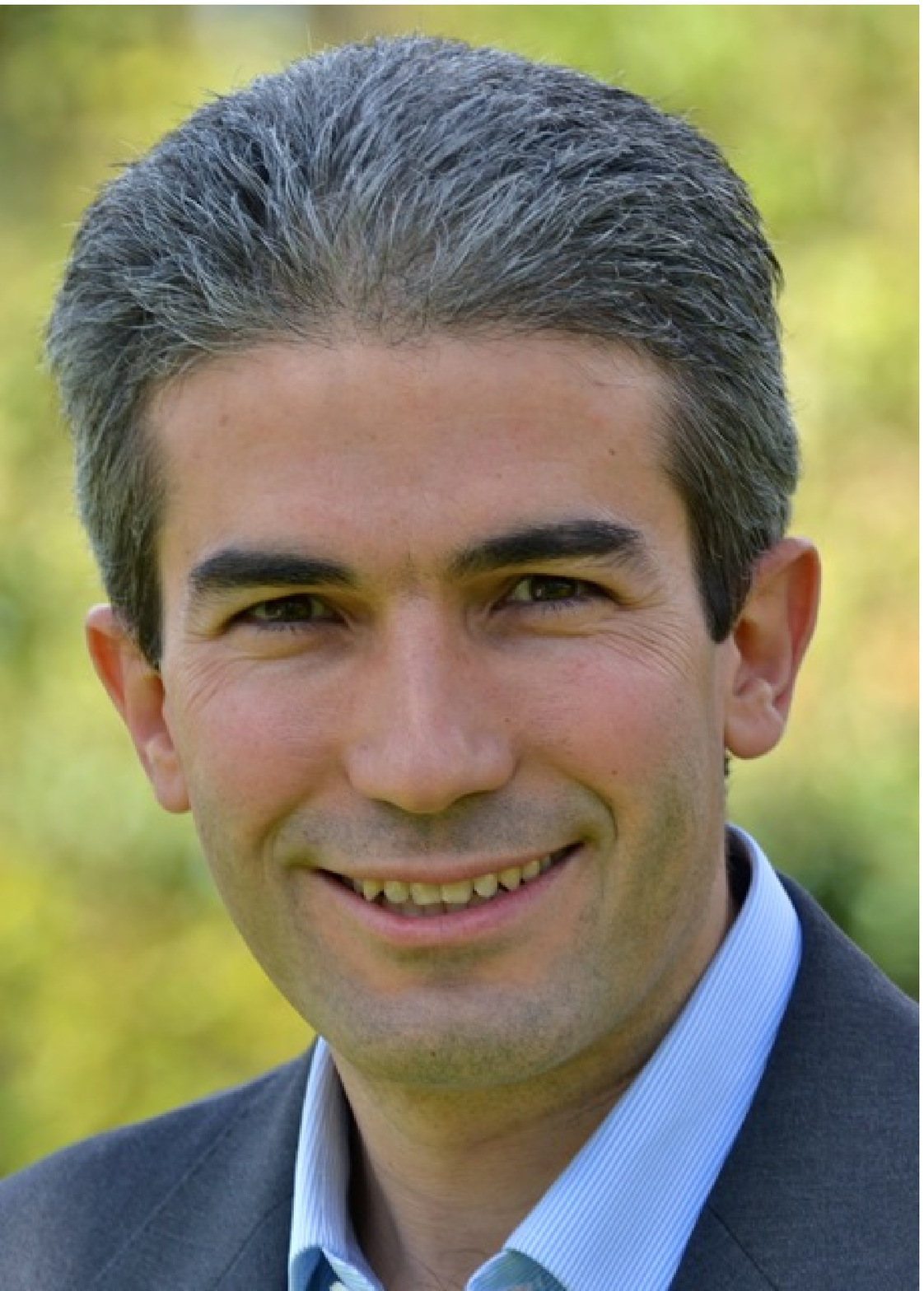}}]{Francesco Bullo} (S'95--M'99--SM'03--F'10) is a Professor with the Mechanical Engineering Department and the Center for Control, Dynamical Systems and Computation at the University of California, Santa Barbara. His main research interests are network systems and distributed control
with application to robotic coordination, power grids and social networks. He is the coauthor of "Geometric Control of Mechanical Systems" (Springer, 2004, 0-387-22195-6) and "Distributed Control of Robotic Networks" (Princeton, 2009, 978-0-691-14195-4).  He received the 2008 IEEE CSM Outstanding Paper Award, the 2010 Hugo Schuck Best Paper Award, the 2013 SIAG/CST Best Paper Prize, and the 2014 IFAC Automatica Best Paper Award. 
\end{IEEEbiography}

\end{document}